\RequirePackage{fix-cm}
\documentclass[smallextended]{svjour3}       
\smartqed  
\usepackage{graphicx}
\usepackage{algorithmic,algorithm}
\usepackage{multirow}
\usepackage{enumerate}
\usepackage{amscd}
\usepackage{lscape}
\usepackage{indentfirst}
\usepackage{latexsym}
\usepackage{amsmath, amsfonts, amssymb,mathrsfs}
\usepackage{verbatim}
\usepackage{bm}
\begin{document}

\title{Randomized and Fault-tolerant Method of Subspace Corrections \thanks{The work of Hu was partially supported by NSF (DMS-1620063). The work of Xu and Zikatanov was supported in part by NSF (DMS-1522615)}
\thanks{On behalf of all authors, the corresponding author states that there is no conflict of interest. }
}

\titlerunning{Randomized Method of Subspace Corrections}        

\author{Xiaozhe Hu \and 
        	    Jinchao Xu  \and 
	    Ludmil T. Zikatanov}

\authorrunning{X.~Hu, J.~Xu, and L.~.T.~Zikatanov} 

\institute{Xiaozhe Hu \at
              Department of Mathematics, Tufts University, Medford, MA 02155, United States \\
              \email{Xiaozhe.Hu@tufts.edu}           
           \and
           Jinchao Xu \at
            Department of Mathematics, The Pennsylvania State University, University Park, PA16802, United States\\
            \email{xu@math.psu.edu}
            \and
            Ludmil T. Zikatanov
            \at
            Department of Mathematics, The Pennsylvania State University, University Park, PA16802, United States\\
            \email{ludmil@psu.edu}
}

\date{Received: date / Accepted: date}

\maketitle

\begin{abstract}
  In this paper, we consider the iterative method of
  subspace corrections with random ordering.  We prove identities for
  the expected convergence rate, which can provide sharp estimates for
  the error reduction per iteration.  We also study the fault-tolerant
  feature of the randomized successive subspace correction method by
  simply rejecting all the corrections when error occurs and show that the
  results iterative method converges with probability one.  Moreover,
  we also provide sharp estimates on the expected convergence rate for the
  fault-tolerant, randomized, subspace correction method.
\end{abstract}

\section{Introduction}
In this paper, we consider iterative methods for solving the following
model problem: Given $f\in V$, find $u\in V$ such that
\begin{equation} \label{eqn:Au=f}
Au = f,
\end{equation}
where $V$ is a Hilbert space and $A: V \mapsto V$ is a symmetric
positive definite (SPD) linear operator.  The class of iterative
methods we are interested fall into the category of the so-called the
methods of subspace corrections (MSC) which have been widely studied
in the past several decades, see
\cite{Xu.J1992,Griebel.M;Oswald.P.1995a,Xu.J;Zikatanov.L2002}.  MSC is
a general framework for linear iterative methods for the solution of
linear problems in Hilbert spaces. Many well-known iterative methods
can be viewed in the MSC framework and, therefore, can be studied
using the general theory of MSC framework, for example, multigrid (MG)
method
\cite{Hackbusch.W.1985a,Bramble.J1993,Trottenberg.U;Oosterlee.C;Schuller.A2001}
and domain decomposition (DD) method
\cite{Quarteroni.A;Valli.A1999,Toselli.A;Widlund.O2005}.

There are basically two kinds of MSC depending on how the subspaces
are corrected. Parallel subspace corrections (PSC) method corrects all
the subspaces simultaneously while the successive subspace corrections
(SSC) method corrects one after another.  The standard SSC method
traverses the subspace problems in a fixed order, but one of the
interesting features of the SSC method is that the ordering need not
be fixed from the start and it can be chosen dynamically during the
iterations.  A classical example in this direction is the greedy
ordering algorithm for Gauss-Seidel method by Southwell
\cite{Southwell.R.1946a}.  Recently, the effects of the greedy
ordering on the convergence of the multiplicative Schwarz method have
been studied in \cite{Griebel.M;Oswald.P.2012a}.  Other algorithms,
such as randomized Kaczmarz iterative method have been recently
studied in detail (see
\cite{Liu.J;Wright.S.2013a,Mansour.H;Yilmaz.O.2013a,Eldar.Y;Needell.D.2011a,Needell.D.2010a,Strohmer.T;Vershynin.R.2009a,Oswald.P;Zhou.W2015a}). A
randomized Schwarz method has been discussed in
\cite{Griebel.M;Oswald.P.2012a}, and a randomized coordinate decent
methods for certain class of convex optimization problems was in the
focus of several recent works~\cite{Richtarik.P;Takac.M.2014a,Leventhal.D;Lewis.A.2010a,Nesterov.Y.2012a}.

One of our main results is the proof of an identity for the
\emph{expected} error reduction in energy norm per iteration step of
the randomized SSC method.  In addition, we propose and analyze the
convergence of a novel SSC method with $J$ subspaces in which the
ordering of the subspace corrections is chosen every $J$ iterations by
randomly selecting a permutation of $\mathbb{J}=\{1,\ldots,J\}$.  We
further provide a generalization of the XZ-identity
\cite{Xu.J;Zikatanov.L2002} which applies to the error reduction rate
in energy norm for such randomized SSC method. Next, we consider a
special feature of this, namely, its convergence in case of hardware
and/or software failures. On one hand, when such error occur, there is
no guarantee that the iterative method can produce a reasonable
approximation of the solution.  On the other hand, for many PDE-based
applications, solving the linear system of equations dominates the
overall simulation time (more than $80 \%$ of the simulation time for
large-scale simulations). Therefore, the development and analysis of
fault-tolerant linear solvers with low overhead is an important and
urgent issue for improving the overall reliability of the huge pool of
PDE-based applications. The standard approaches for constructing
fault-tolerant iterative methods usually belong to the so-called ABFT
(Algorithm-Based Fault Tolerance) category and basic linear and fault
tolerant versions of nonlinear iterative methods, such as successive over-relaxation (SOR) method, conjugate gradient (CG) method, and general minimal residual (GMRes) method have been studied in~\cite{Roy-Chowdhury.A;Bellas.N;Banerjee.P.1996a,Roy-Chowdhury.A;Banerjee.P.1993a,Hoemmen.M;Heroux.M.2011a,Shantharam.M;Srinivasmurthy.S;Raghavan.P.2012a}.
Another approach, proposed in~\cite{Stoyanov.M;Webster.C.2013a} relies
on rejecting large hardware error propagation and improves the
resilience of iterative methods with respect to silent errors.  More
recently, in \cite{Cui.T;Xu.J;Zhang.C2017a}, an intrinsic fault/error
tolerant feature of the MSC has been explored.  The idea is based on
introducing redundant subspaces and working with specially designed
mappings between subspaces and processors.

Our results also show how the randomization can be used to improve the
reliability of the SSC method in this paper. We built our
fault-tolerant, randomized SSC method on a procedure which rejects the
faulty subspace corrections when errors occur.  Basically, we only
update the solution when there is no error and, naturally, we are able
to show that this simple procedure, together with randomization,
converges almost surely (with probability $1$).  Our results
demonstrate the potential of the SSC method as a natural
fault-tolerant iterative method and provide theoretical justification
of the usage of the SSC method in improving the reliability of
long-running large-scale PDE applications.

The reminder of the paper is organized as follows.  We recall the PSC
method, the SSC method, the XZ-identity, and some basic notions from
probability theory in Section \ref{sec:msc}.  In Section
\ref{sec:alg}, we describe the randomized SSC method and its
fault-tolerant variant.  Section \ref{sec:converge} presents our main
results, i.e. the sharp identity estimates of the convergence rate and
almost sure convergence of the proposed SSC methods.  At the end, we
give some remarks in Section \ref{sec:conclu} to conclude the paper.

\section{Preliminaries}
In this section, we introduce the notation and review some basic
results and definitions from the theory of the subspace correction
methods and basic probability. 

\subsection{Method of Subspace Corrections} \label{sec:msc}
In this subsection, we recall the standard method of subspace corrections.  We consider a decomposition of the vector space $V$ which consists subspaces $V_i \subset V$, $i=1, 2, \cdots, J$, such that 
\begin{equation} \label{eqn:space-decomp}
V = \sum_{i=1}^J V_i. 
\end{equation}
This means that, for each $v \in V$, there exist $v_i\in V_i$, $i = 1, 2, \cdots, J$, such that $v=\sum_{i=1}^Jv_i$.  This representation of $v$ may not be unique in general, namely \eqref{eqn:space-decomp} is not necessarily a direct sum.

For each $i$, we define $Q_i,P_i:  V \mapsto V_i$ and $A_i: V_i \mapsto V_i$ by
\begin{equation}\label{eqn:Qi_Pi} 
(Q_iu,v_i)=(u,v_i),\quad (P_iu,v_i)_A=(u,v_i)_A,
\quad \forall \ u \in V, \ v_i\in V_i, 
\end{equation}
and
\begin{equation}\label{eqn:Ai}
(A_iu_i,v_i)=(Au_i,v_i),\quad \forall \ u_i, v_i\in V_i.
\end{equation}
$Q_i$ and $P_i$ are both orthogonal projections and $A_i$ is the restriction of $A$ on $V_i$ and is SPD.  

It follows from the definition that
\begin{equation}\label{eqn:AiPi=QiA}
A_iP_i=Q_iA.
\end{equation}
Indeed, $\forall \, u,v\in V$, we have $(Q_iAu,v)=(Au,Q_iv)=(u,Q_iv)_A=(P_iu,Q_iv)_A=(A_iP_iu,Q_iv)=(A_iP_iu,v)$, therefore $A_iP_i=Q_iA$. 

Since $V_i\subset V$, we may consider the natural inclusion $I_i:V_i \mapsto V$ defined by 
\begin{equation}\label{eqn:Ii}
(I_iu_i, v) =(u_i, v),\quad\forall \, u_i\in V_i. \quad v \in V.
\end{equation}
We notice that $Q_i=I_i^T$ as $(Q_iu,v_i)=(u,v_i)=(u,I_iv_i)=(I_i^Tu,v_i)$.  Similarly, we have $P_i=I_i^*$, where $I_i^*$ is the transpose of $I_i$ with respect to the inner product $(\cdot,\cdot)_A$ induced by $A$.

If $u$ is the solution of \eqref{eqn:Au=f}, then
\begin{equation}\label{eqn:Aiui=fi}
A_i u_i = f_i,
\end{equation}
where $u_i = P_i u$ and $f_i = Q_i f$.  \eqref{eqn:Aiui=fi} can be
viewed as the restriction of \eqref{eqn:Au=f} on the subspace $V_i$,
$i=1,2,\cdots, J$.  MSC solves these subspace equations
\eqref{eqn:Aiui=fi} iteratively.  In general, these
subspace equations are solved approximated.  More precisely, we
introduce a non-singular operator $R_i: V_i \mapsto V_i$,
$i=1,2,\cdots, J$, which assumed to be an approximation of $A_i^{-1}$
in certain sense.  And then the subspaces equation \eqref{eqn:Aiui=fi}
are solved approximated by $u_i \approx \hat{u}_i = R_if_i$.

As we pointed out in the introduction, there are two major type MSC,
depending on how the error is corrected by solving subspace
problems. These are \emph{the parallel subspace corrections (PSC) method} 
and \emph{the successive subspace corrections (SSC) method}.
The PSC method is similar in nature to the classical Jacobi method,
where the subspace equations are solved in parallel as in~Algorithm~\ref{alg:PSC}.
\begin{algorithm}
\caption{Parallel Subspace Correction Method} \label{alg:PSC}
\begin{algorithmic}[1]
\STATE Compute the residual by $r^m = f - A u^m$,
\STATE Approximately solve the subspace equations $A_i e_i = Q_i r^m$ by $\hat{e}_i = R_i Q_i r^m$ in parallel,
\STATE Update the iteration by $u^{m+1} = u^m + \sum_{i=1}^J I_i \hat{e}_i$.
\end{algorithmic}
\end{algorithm}
From the definitions in Algorithm \ref{alg:PSC}, it is easy to see that
\begin{equation*}
u^{m+1} = u^{m} + B_a (f - Au^m),
\end{equation*}
with 
\begin{equation}\label{eqn:Ba}
B_a = \sum_{i=1}^J I_i R_i Q_i = \sum_{i=1}^J I_i R_i I_i^t,
\end{equation}
which is the operator corresponds to the PSC method. 

The SSC method is similar to the classical Gauss-Seidel iterative
method and the error is corrected successively in every subspace as
outlined in Algorithm~\ref{alg:SSC}.
\begin{algorithm}
\caption{Successive Subspace Correction Method} \label{alg:SSC}
\begin{algorithmic}[1]
\STATE Compute the residual by $r^m = f - A u^m$,
\STATE Set $v^0 = u^m$,
\FOR {$k = 0 \to J-1$}
\STATE $v^{k+1} = v^{k} + R_k Q_k (f - A v^{k})$,
\ENDFOR
\STATE Update $u^{m+1} = v^J$.
\end{algorithmic}
\end{algorithm}

Let us define $T_i = R_iQ_iA$, and we note that $T_i:V_i \mapsto V_i$
is symmetric with respect to $(\cdot, \cdot)_A$, nonnegative definite,
and satisfies $T_i = R_i A_i P_i$. Moreover, $T_i = P_i$ if
$R_i=A_i^{-1}$.  Using this notation, we have
\begin{equation}\label{eqn:B_m}
I - B_m A = (I - T_J) (I - T_{j-1}) \cdots (I - T_i),
\end{equation} 
where $B_m$ is the operator approximating $A^{-1}$ which
corresponds to the SSC method. 

In this paper, we focus on the randomized and fault-tolerant versions
of the SSC method $B_m$.  For its convergence analysis, we have the
following well-known XZ-identity \cite{Xu.J;Zikatanov.L2002}.
\begin{theorem}[XZ-identity]\label{thm:XZ-identity}
Assume that $B_m$ is defined by the SSC method (Algorithm \ref{alg:SSC}), then we have
\begin{equation}\label{eqn:xz-identity}
\|I-B_mA\|_A^2=1-\frac{1}{1+c_0}=1-\frac{1}{c_1},
\end{equation}
where
\begin{equation}\label{eqn:xz-c0}
c_0=\sup_{\|v\|_A=1}\inf_{\sum
  v_i=v}\sum_{i=1}^J\|R_i^tw_i\|_{{\overline R_i^{-1}}}^{2}.
\end{equation}
with $w_i=A_iP_i\sum_{j\geq i}v_j-R_i^{-1}v_i$, 
and
\begin{equation}\label{eqn:xz-c1}
c_1=\sup_{\|v\|_A=1}\inf_{\sum v_i=v}\sum_{i=1}^J
\left\|\overline{T_i}^{-1/2}\left(v_i+T_i^*P_i\sum_{j>i}v_j\right)\right\|_{A}^2,
\end{equation}
with 
$$
\overline T_i=T_i+T_i^*-T_i^*T_i \text{ and } T_i=R_iA_iP_i,\quad 1\leq
i\leq J.
$$
\end{theorem}

\begin{corollary}
  In case that the subspace problems are solved exactly,
  i.e. $R_i=A_i^{-1}$, the X-Z identity \eqref{eqn:xz-identity} holds
  with
\begin{equation}
c_0=\sup_{\|v\|_A=1}\inf_{\sum v_i=v}\sum_{i=1}^J\left\| P_i\sum_{j> i}v_j\right\|_{A_i}^{2}.
\end{equation}
and
\begin{equation}
c_1=\sup_{\|v\|_A=1}\inf_{\sum v_i=v}\sum_{i=1}^J
\left\|P_i\left(\sum_{j\geq i}v_j\right)\right\|_{A}^2.
\end{equation}
\end{corollary}

\section{Randomized and Fault-tolerant SSC} \label{sec:alg}
Traditionally, the SSC method visits each subspace in a pre-determined ordering, i,e, it solves subspace problems one by one in a fixed, problem-independent order.  Here we consider to choose the ordering randomly, which is a key component of the algorithms.  Another component we introduce into the SSC method is the fault-tolerant ability enabled by randomization.  In this section, we formulate those algorithms and their convergence analysis are discussed in the next section. 

\subsection{Randomized SSC}
In the randomized SSC method, we randomly choose the next subspace
in which the error needs to be corrected. We randomly choose the subspace, 
according to certain probability distribution as in Algorithm \ref{alg:randomSSC-1}.

\begin{algorithm}
\caption{SSC method with random ordering (Version 1)} \label{alg:randomSSC-1}
\begin{algorithmic}[1]
\STATE Randomly choose an index $i \in \{ 1, 2, \cdots, J \}$ with probability $p_i = \frac{1}{J}$,
\STATE $u^{k+1} = u^k + R_i Q_i (f- Au^k)$
\end{algorithmic}
\end{algorithm}

As discussed in \cite{Griebel.M;Oswald.P.2012a}, the cost of randomly
picking $i$ does not exceed $\mathcal{O}( \log J)$ and each update
in the SSC method can be done in $\mathcal{O}(N)$ operations where $N$ is
the dimension of the vector space $V$.  Therefore, the overall
computational cost of the randomized SSC method is comparable to the
standard SSC method which is a very desirable feature.

Note that in Algorithm \ref{alg:randomSSC-1}, there is no guarantee
that all the $J$ subspaces are all corrected in $J$ iterations.
Therefore, we propose the second version randomized SSC method, such
that the $J$ subspaces are guaranteed to be corrected within $J$
iterations by randomly choosing the ordering in which the error is
corrected.  To do this, we first consider the set of all permutations
of $\mathbb{J} = \{ 1,2, \cdots, J \} $. Then a permutation of
$\mathbb{J}$ is any bijective mapping
$\sigma : \mathbb{J} \mapsto \mathbb{J}$.  The idea is to randomly
choose a permutation $\sigma$ from the set of permutations and apply
the SSC following the correction order as specified by $\sigma$. We
have the randomized SSC method presented in Algorithm
\ref{alg:randomSSC-2}.
\begin{algorithm}[h!]
\caption{SSC method with random ordering (Version 2)} \label{alg:randomSSC-2}
\begin{algorithmic}[1]
\STATE Compute the residual by $r^m = f - A u^m$,
\STATE $v……0 = u^m$,
\STATE Randomly choose a permutation $\sigma$ of the indexes $\mathbb{J} = \{  1,2, \cdots, J \}$ with probability $\frac{1}{J!}$,
\FOR {$k = 0 \to J-1$}
\STATE $v^{k+1} = v^{k} + R_{\sigma(k)} Q_{\sigma(k)} (f - A v^{k})$,
\ENDFOR
\STATE Update the iteration by $u^{m+1} = v^J$.
\end{algorithmic}
\end{algorithm}

We need to note that, in Algorithm \ref{alg:randomSSC-2}, the cost of randomly picking the permutation $\sigma$ is $\mathcal{O}(\log J!) = \mathcal{O}(J \log J)$ which is expensive than Algorithm \ref{alg:randomSSC-1}.  But each update still can be done in $\mathcal{O}(N)$ operations.  It is reasonable to assume that $J = \mathcal{O}(N)$ at the worst case, therefore, the overall computational cost of Algorithm \ref{alg:randomSSC-2} is $\mathcal{O}(N \log N)$ which is slightly expensive than traditional SSC method and Algorithm \ref{alg:randomSSC-1}.

\subsection{Fault-tolerant Randomized SSC}
Another feature pertaining to this randomized SSC method is its
fault-tolerance.  During the iterative process, the correction or
update may fail due to hard and/or soft errors, which, if not handled
correctly, may result in a stagnating iterative method.  We propose a
simple approach which can handle all such scenarios (see Algorithm
\ref{alg:randomSSC-fault}).  Basically, we do not update the
approximation to the solution during the iterations when error
occurs. The randomization of the ordering helps to guarantee that such
simple treatment leads to theoretically convergent iterative method.
\begin{algorithm}
\caption{Fault-tolerant SSC method with random ordering} \label{alg:randomSSC-fault}
\begin{algorithmic}[1]
\IF {error occurs}
\STATE $u^{k+1} = u^k$, 
\ELSE 
\STATE Randomly choose an index $i \in \{ 1, 2, \cdots, J \}$ with probability $p_i = \frac{1}{J}$,
\STATE $u^{k+1} = u^k + R_i Q_i (f- Au^k)$.
\ENDIF
\end{algorithmic}
\end{algorithm}

Similar to Algorithm \ref{alg:randomSSC-1}, the cost of randomly picking $i$ is $\mathcal{O}(\log J)$ and each correction costs $\mathcal{O}(N)$.  Therefore, the overall cost of Algorithm \ref{alg:randomSSC-fault} is comparable to the cost of the traditional SSC method and Algorithm \ref{alg:randomSSC-1}.

\section{Convergence Analysis} \label{sec:converge}
In this section, we discuss the convergence analysis of the randomized
and fault-tolerant SSC methods (Algorithm \ref{alg:randomSSC-1} -
\ref{alg:randomSSC-fault}).  We want to emphasize that, instead of
usual upper bound estimation, we present identities to estimate the
convergence rate of the randomized and fault-tolerant SSC methods. We
note that in the analysis below we relate the expected convergence
rate of the SSC method, to the quality of the PSC preconditioner,
which is independent of the ordering. This is not surprising and shows
rigorously the fact that the expected (average) convergence rate of an
SSC method is also independent of the ordering.

\subsection{Convergence Rate of the Randomized SSC}
First, we consider Algorithm \ref{alg:randomSSC-1} and the main result
is stated in the following theorem.  Here, we use $B_a$ to denote the
operator corresponding to the PSC method with $\bar{R}_i = R_i^t + R_i
- R_i^tA_iR_i$ as the inexact subspace solver.
\begin{theorem} \label{thm:randomSSC-1}
The Algorithm \ref{alg:randomSSC-1} converges with the expected error decay rate,
\begin{equation} \label{eqn:randomSSC-1}
E(\|u - u^{k+1}\|_A^2) = (1 -  \frac{\delta_k}{J} ) E( \| u-u^{k}\|_A^2) = \prod_{\ell=0}^{k}(1-\frac{\delta_{\ell}}{J})\| u - u^0 \|_A^2,
\end{equation}
where $\delta_k = \frac{E( (B_a A e^k, e^k )_A) }{E((e^k, e^k)_A)} > 0$ and $e^k= u - u^k$. Moreover, if $\| I - T_i \|_A < 1$, then $\delta_k < J$.
\end{theorem}
\begin{proof}
It is easy to see that, given $i$, we have
\begin{align*}
\| e^{k+1} \|_A^2 = \| (I - T_i) e^k \|_A^2 = ( \left( I - \bar{T}_i \right) e^k, e^k )_A,
\end{align*}
where $\bar{T}_i$ is the symmetrized version of $T_i$.  According to
the way we pick the index $i$, at iteration $k$, consider the probability space
$(\bm{\Omega}_k, \bm{\mathcal{F}}_k, P)$ where 
\begin{equation*}
\bm{\Omega}_k = \underbrace{\Omega \times \Omega \times \cdots \times \Omega}_{k},  \quad \Omega:= \{1, 2, \cdots, J\},
\end{equation*}
$\bm{\mathcal{F}}_k$ contains all the subset of $\bm{\Omega}_k$, and $P(a) = |a| / |\bm{\Omega}_k|$, $a \in \bm{\mathcal{F}}_k$.  We define random variables $S_k(x): \bm{\Omega}_k \mapsto \mathbb{R}^k$ such that,
\begin{equation*}
S_{k}(\omega_k) = \omega_k, \ \omega_k = (i_1, i_2, \cdots, i_k) \in \bm{\Omega}_k, \ i_m \in \Omega, \ m = 1, \cdots, k.
\end{equation*}
Because we choose an index $i \in \{  1,2, \cdots, J \}$ uniformly with probability $1/J$ and picking $i$ is independent of the iteration number $k$, we have,
\begin{equation*}
P( S_{k+1} = \omega_{k+1} \, | \, S_k = \omega_k)   
=
\begin{cases}
\frac{1}{J}, & \ \text{if} \ \omega_{k+1} = (\omega_k, i), \ i= 1, 2, \cdots, J \\
0, & \text{otherwise}
\end{cases}
\end{equation*}
Then we define functions 
\begin{equation*}
g_{k}(\omega_k) := \| (I - T_{i_k}) (I - T_{i_{k-1}}) \cdots (I - T_{i_1}) e^0 \|_A^2, \quad \omega_k = (i_1, i_2, \cdots, i_k)
\end{equation*}
 and the compositions of $g_{k}$ and $S_{k}$ define other random variables which we denote by
\begin{equation*}
X_{k} = g_{k}\circ S_{k},
\end{equation*}
and, for $\omega_{k} \in \bm{\Omega}_{k}$
\begin{equation*}
X_{k}(\omega_{k}) = g_{k}(S_{k}(\omega_{k})) = \| (I - T_{i_k}) (I - T_{i_{k-1}}) \cdots (I - T_{i_1}) e^0 \|_A^2 := \| e^k \|_A^2,
\end{equation*}

Then we compute the conditional expectation as following,
\begin{align*}
E(X_{k+1} \, | \, S_{k})(\omega_k) &= E(X_{k+1} \, | \, S_{k} = \omega_k) \\
&= \sum_{\omega_{k+1} \in \bm{\Omega}_{k+1}} X_{k+1}(\omega_{k+1}) P( S_{k+1} = \omega_{k+1} \, | \, S_k = \omega_k ) \\
	& = \sum_{i=1}^J \| (I - T_i)e^{k} \|_A^2 \, \frac{1}{J}  \\
	& = \sum_{i=1}^{J} \frac{1}{J} ( (I-\bar{T}_i) e^k, e^k )_A  \\
	& = \| e^k \|^2_A - \frac{1}{J} ( \sum_{i}^J \bar{T}_i e^k, e^k)_A  \\
	&= \| e^k \|_A^2 - \frac{1}{J} (B_a A e^k, e^k)_A.
\end{align*}
Apply $E(X) = E(E(X|Y))$, use the linearity of the expectation, and let $X= X_{k+1}$, $Y = S_k$, we then have
\begin{align*}
E(\| e^{k+1} \|^2_A) = E(X_{k+1})  & = E(E(X_{k+1}|S_k))  \\
&= E( \| e^k \|^2_A - \frac{1}{J} ( B_a A e^k, e^k)_A )   \\
& =  E(\|e^k\|_A^2) - \frac{1}{J} E( (B_a A e^k, e^k)_A ) \\
& = \left(  1 - \frac{1}{J} \frac{E( (B_a A e^k,e^k)_A)}{ E( (e^k, e^k)_A ) }  \right) E(\|e^k\|_A^2)
\end{align*}
Note that, if $\| I - T_i \|_A < 1 $, we have $(B_a A e^k, e^k)_A = (\sum_{i=1}^J \bar{T}_i e^k, e^k)_A < J  (e^k, e^k)_A$ which implies that $ E( (B_aAe^k, e^k)_A ) < J E( (e^k, e^k)_A) $, i.e. $\delta_k < J$.  This completes the proof. 
\end{proof}

\begin{remark} \label{rem:delta-randomSSC-1}
Now we discuss about the constant $\delta_k$, 
we have
\begin{equation*}
\lambda_{\min}(B_aA) \| e^k \|_A^2 \leq (B_a A e^k, e^k )_A  \leq  \lambda_{\max}(B_aA) \| e^k \|_A^2.
\end{equation*}
Use the linearity and monotonicity of expectation, we have
 \begin{equation*}
\lambda_{\min}(B_aA) E(\|e^k\|_A^2) \leq  E((B_a A e^k, e^k )_A) \leq \lambda_{\max}(B_aA) E(\|e^k\|_A^2).
\end{equation*}

Therefore, we have
\begin{equation*}
\lambda_{\min}(B_aA) \leq \delta_k \leq \lambda_{\max} (B_aA),
\end{equation*}
and
\begin{equation*}
1 - \frac{\delta_k}{J} \leq 1 - \frac{\lambda_{\min}(B_aA)}{J}.
\end{equation*}
\end{remark}

\begin{remark}
  After $J$ steps, we have $E(\| e^{J} \|^2_A) \leq \left(1 -
    \frac{\lambda_{\min}(B_a A)}{J}\right)^J \| e^0 \|_A^2$ and the
  energy error reduction is bounded by $\left(1 -
    \frac{\lambda_{\min}(B_a A)}{J}\right)^J \le
  \exp\left({-\lambda_{\min}(B_a A)}\right)$.
\end{remark}

\begin{remark}
 Taking the multigrid method as an example, if we use deterministic
 approach, for any given $i$, we can only get a reduction that is
 given by a smoother (such as Gauss-Seidel), namely
 \begin{equation*}
   \|e^{k+1}\|_A\le (1-ch^2)\|e^k\|_A.
 \end{equation*}
But for randomized method, we have 
 \begin{equation*}
    E(\|e^{k+1}\|_A)\le (1-c/|\log h|) E(\|e^k\|_A)
   \quad \mbox{since } J=\mathcal O(|\log h|).
 \end{equation*}
 This is a significantly improvement.
\end{remark}
\begin{remark}
  Here we choose next subspace used for correction
  based on a uniformly distributed random variable with probability
  constant probability $p_i = \frac{1}{J}$.  However, one can use
  other choices as long as the probability does not depend on the
  iteration number $k$ and similar results can be derived.  Same
  statement also applied to the theoretical results that follow.
\end{remark}
\begin{equation*} 
E(\|u - u^{kJ}\|_A^2) \le
\delta^kE(\| u - u^0 \|_A^2),\quad \delta= e^{-\lambda_{\min}(B_aA)}
\end{equation*}

A special case of Algorithm \ref{alg:randomSSC-1} is that the subspace
corrections are exact, i.e, $R_i = A_i^{-1}$.  And the following
corollary is a direct consequence of Theorem \ref{thm:randomSSC-1}.

\begin{corollary} \label{thm:randomSSC-1-exact} Assume that the
  probabilities $p_i$, $i\in \mathbb{J}$ are independent of the iteration number $k$.
  Then Algorithm \ref{alg:randomSSC-1} with $R_i = A_i^{-1}$ converges
  with the following expected error reduction:
\begin{equation} \label{eqn:randomSSC-1-exact}
E(\|u - u^{k+1}\|_A^2) = (1 -  \frac{\delta_k}{J} ) E( \| u-u^{k}\|_A^2) = \prod_{\ell=0}^{k} (1 - \frac{\delta_{\ell}}{J}) \| u - u^0 \|_A^2,
\end{equation}
where $0 < \delta_k = \frac{E\left(\sum_{i=1}^J ||P_i e^k ||_A^2\right)}{E(||e^k||_A^2)} < J$.
\end{corollary}

Next theorem shows that the randomized SSC method converges almost
surely, i.e. converges with probability $1$, if all the subspace corrections are convergent.
\begin{theorem} \label{thm:randomSSC-1-p1}
If $\| I - T_i \|_A < 1$ for $i=1,2, \cdots, J$, then Algorithm \ref{alg:randomSSC-1} converges almost surely or with probability $1$, i.e., $\| e^k \|_A^2 \overset{a.s}{\longrightarrow} 0$.
\end{theorem}
\begin{proof}
In order to show the almost sure convergence, we need to show that $\sum_{k=1}^{\infty} P(\| e^k \|_A^2  \geq \varepsilon) < + \infty $ for any $\varepsilon > 0$.  Note that, by Markov's inequality, we have
\begin{align*}
P(\| e^k \|_A^2  \geq \varepsilon) \leq \frac{1}{\varepsilon} E(\| e^k \|_A^2) = \frac{1}{\varepsilon} \prod_{\ell=0}^{k-1}
\left( 1 - \frac{\delta_{\ell}}{J} \right) \| e^0 \|_A^2,
\end{align*}
Since $\| I - T_i \|_A < 1$ and according to Remark \ref{rem:delta-randomSSC-1}, we have $\lambda_{\min}(B_aA)<\delta_{\ell} < \lambda_{\max}(B_aA) < J$ and then,
\begin{align*}
\sum_{k=1}^{\infty}  \frac{1}{\varepsilon} \prod_{\ell=0}^{k-1} \left( 1 - \frac{\delta_{\ell}}{J} \right) \| e^0 \|_A^2 & \leq \frac{1}{\varepsilon} \| e^0 \|_A^2 \sum_{k=1}^{\infty} \left( 1 - \frac{\lambda_{\min}(B_aA)}{J}  \right)^{k-1} \\
& = \frac{1}{\varepsilon} \| e^0 \|_A^2 \frac{J}{\lambda_{\min}(B_aA)}  < + \infty.
\end{align*}
Therefore, we have that the series $\sum_{k=1}^{\infty} P(\| e^k \|_A^2  \geq \varepsilon)$ converges. 
\end{proof}

Next, we discuss the convergence rate for Algorithm \ref{alg:randomSSC-2}.  We introduce $B_\sigma$ to denote the corresponding operator for $J$ iterations using permutation $\sigma$, i.e. 
\begin{equation*}
I - B_\sigma A = (I - T_{\sigma(J)}) (I - T_{\sigma(J-1)}) \cdots (I- T_{\sigma(1)}).
\end{equation*}
We have the following theorem
\begin{theorem} Consider $B_\sigma$ defined by Algorithm \ref{alg:randomSSC-2}, we have
\begin{equation*}
E( \| I - B_\sigma A \|_A^2) = 1 - \frac{1}{J!} \sum_{i=1}^{J!} \frac{1}{c_{\sigma_i}}.
\end{equation*}
\end{theorem}
\begin{proof}
According to XZ-identity \eqref{eqn:xz-identity}, for a permutation $\sigma_i$, we have
\begin{equation*}
|| I - B_{\sigma_i} A||_A^2 = 1 - \frac{1}{c_{\sigma_i}}.
\end{equation*}
Here, we need to introduce a different probability space $(\Omega, \mathcal{F}, P)$ where $\Omega = \{ \sigma_1, \sigma_2, \cdots, \sigma_{J!} \}$, i.e., the set of all possible permutations, $\mathcal{F}$ is again the $\sigma$-algebra of $\Omega$, and the probability $P$ is defined by $P(a) = |a|/J!$, $a \in \mathcal{F}$.  Based on this probability space, 
we define a random variable $X:  \Omega \mapsto \mathbb{R}$ as following,
\begin{equation*}
X(\sigma_i) =  \| I - B_{\sigma_i} A \|_A^2 = 1 - \frac{1}{c_{\sigma_i}} =: x_{\sigma_i}, \quad i = 1, 2, \cdots, J!
\end{equation*}  
$X$ is a discrete random variable an its expectation can be computed as following
\begin{align*}
E( \| I - B_\sigma A \|_A^2) = E(X) 
	& = \sum_{i=1}^{J!} x_{\sigma_i} P(X = x_{\sigma_i})  \\
	& = \sum_{i=1}^{J!} \left( 1 - \frac{1}{c_{\sigma_i}} \right) \, \frac{1}{J!}  \\
	& = 1 - \frac{1}{J!} \sum_{i=1}^{J!} \frac{1}{c_{\sigma_i}},
\end{align*}
which completes the proof. 
\end{proof}

\begin{remark}
Note that
\begin{align*}
1 - \frac{1}{J!} \sum_{i=1}^{J!} \frac{1}{c_{\sigma_i}} \leq 1 - \frac{1}{c_{\max}},
\end{align*}
where $c_{\max} = \max_{1\leq i \leq J!} c_{\sigma_i}$.  The equality
holds if and only if $c_{\sigma_i} = c_{\max} $ for all $\sigma_i$.
Therefore randomized SSC method Algorithm \ref{alg:randomSSC-2}
``improves'' the convergence rate of the traditional SSC method which
only considers the worst case scenario.
\end{remark}

\subsection{Fault-tolerant Randomized SSC}
In this section, we discuss the convergence rate of the fault-tolerant randomized SSC method (Algorithm \ref{alg:randomSSC-fault}).  The main assumption is that the errors occur with probability $\theta \in [0,1)$.  Next theorem says that the fault-tolerant randomized SSC method converges in expectation.  Again, we use $B_a$ to denote the operator corresponding to the PSC method with $\bar{R}_i = R_i^t + R_i - R_i^tA_iR_i$ as the inexact subspace solver.

\begin{theorem} \label{thm:randomSSC-fault}
Assume that error occurs with probability $\theta \in[0,1)$ which is independent of $k$ and how $i$ is picked, then the Algorithm \ref{alg:randomSSC-fault} converges with the expected convergence rate,
\begin{align} 
E(\|u - u^{k+1}\|_A^2) &= \left(1 -  \frac{(1- \theta)\delta_k}{J} \right) E( \| u-u^{k} \|_A^2) \nonumber \\
& = \prod_{\ell=0}^{k}\left(1 -  \frac{(1- \theta)\delta_{\ell}}{J} \right) || u-u^{0}||_A^2, \label{eqn:randomSSC-fault}
\end{align}
where  $\delta_k =  \frac{ E((B_a A e^k, e^k )_A)}{E( (e^k, e^k)_A)} >0$ and $e^k = u - u^k$. Moreover, if $\| I - T_i \|_A < 1$, then $\delta_k < J$.
\end{theorem}

\begin{proof}
Note that, if there is no error (with probability $1-\theta$), for a given $i$, we have
\begin{equation*}
\| e^{k+1} \|_A^2  = \| (I- T_i) e^k \|_A^2 = ( \left( I - \bar{T}_i \right) e^k, e^k )_A,
\end{equation*}
otherwise, we have
\begin{equation*}
\| e^{k+1} \|_A^2 = \| e^{k} \|_A^2.
\end{equation*}
Let us consider the probability space
$(\bm{\Omega}_k, \bm{\mathcal{F}}_k, P)$ where 
\begin{equation*}
\bm{\Omega}_k = \underbrace{\Omega \times \Omega \times \cdots \times \Omega}_{k},  \quad \Omega:= \{0, 1, 2, \cdots, J\},
\end{equation*}
where $\{ 0 \}$ means error occurs, which corresponds to the case that
the \textbf{if} statement (line 1 of Algorithm
\ref{alg:randomSSC-fault}) branching with ``true'' condition and
following the line after \textbf{then}. The choice $\{ i \}$,
$i = 1, 2, \cdots, J$, means there is no error and that index $i$ is
picked at random. This, latter case corresponds to the \textbf{else}
statement (line 1 of Algorithm \ref{alg:randomSSC-fault}) and index
$i$ is picked (line 4 of Algorithm \ref{alg:randomSSC-fault}).  $\bm{\mathcal{F}}_k$ contains all the subset of $\bm{\Omega}_k$.  The probability $P$ is given in the following way: We set
$P(\emptyset) = 0$.  Next, $P(\{0\})=\theta\in [0,1)$ is the
probability that error occurs. Further, $P(\{ i \})$ is the
probability that there is no error and index $i$ is picked. Assuming
that all hard/soft errors occur independently from how $i$ is picked,
we have $P(\{ i\}) = P( \text{no error} \cap \text{$i$ is picked}) =
P(\text{no error}) \, P(\text{$i$ is picked}) = (1-\theta) (1/J) =
(1-\theta)/ J$.
For other events $a \in \mathcal{F}$, $P(a)$ is defined using the
\emph{countable additivity} property.

As in the proof of Theorem \ref{thm:randomSSC-1}, we define random variables $S_k(x): \bm{\Omega}_k \mapsto \mathbb{R}$ such that,
\begin{equation*}
S_{k}(\omega_k) = \omega_k, \ \omega_k = (i_1, i_2, \cdots, i_k) \in \bm{\Omega}_k, \ i_m \in \Omega, \ m = 1, \cdots, k.
\end{equation*}
Because we pick $i$ is independent of the iteration number $k$, we have,
\begin{equation*}
P( S_{k+1} = \omega_{k+1} \, | \, S_k = \omega_k)   
=
\begin{cases}
\theta, & \ \text{if} \ \omega_{k+1} = (\omega_k, 0), \\
\frac{1-\theta}{J}, & \ \text{if} \ \omega_{k+1} = (\omega_k, i), \ i= 1, 2, \cdots, J \\
0, & \text{otherwise}
\end{cases}
\end{equation*}

Then we define functions 
\begin{equation*}
g_{k}(\omega_k) := \| (I - T_{i_k}) (I - T_{i_{k-1}}) \cdots (I - T_{i_1}) e^0 \|_A^2, \ \omega_k = (i_1, i_2, \cdots, i_k)
\end{equation*}
with $T_0 = 0$ and the compositions of $g_{k}$ and $S_{k}$ define other random variables which we denote by
\begin{equation*}
X_{k} = g_{k}\circ S_{k},
\end{equation*}
and, for $\omega_{k} \in \bm{\Omega}_{k}$
\begin{equation*}
X_{k}(\omega_{k}) = g_{k}(S_{k}(\omega_{k})) = \| (I - T_{i_k}) (I - T_{i_{k-1}}) \cdots (I - T_{i_1}) e^0 \|_A^2 := \| e^k \|_A^2,
\end{equation*}

Therefore, then we compute the conditional expectation as following
\begin{align*}
E(X_{k+1} \, | \, S_{k})(\omega_k) &= E(X_{k+1} \, | \, S_{k} = \omega_k) \\
&= \sum_{\omega_{k+1} \in \bm{\Omega}_{k+1}} X_{k+1}(\omega_{k+1}) P( S_{k+1} = \omega_{k+1} \, | \, S_k = \omega_k ) \\
& =  \theta \| e^k \|_A^2 +  \sum_{i=1}^J \frac{(1-\theta)}{J} \| (I- T_i) e^k \|_A^2 \\
 & =  \theta \| e^k \|^2_A +  \frac{1-\theta}{J} J (e^k, e^k)_A - \sum_{i=1}^J \frac{1-\theta}{J} (\bar{T}_ie^k, e^k)_A\\
 & = \|e^k\|^2_A  -  \frac{1-\theta}{J } (B_a A e^k, e^k)_A.
\end{align*}
Following the proof of Theorem \ref{thm:randomSSC-1}, we apply the
identity $E(X) = E(E(X|Y))$ and use the
linearity of the expectation to derive \eqref{eqn:randomSSC-fault}. This
completes the proof.
\end{proof}

\begin{remark}
we can estimate the constant $\delta_k$ as in Remark \ref{rem:delta-randomSSC-1}, i.e.
\begin{equation*}
\lambda_{\min}(B_aA) \leq \delta_k \leq \lambda_{\max} (B_aA).
\end{equation*}
Therefore, we have
\begin{equation*}
1 - \frac{(1-\theta )\delta}{J} \leq 1 - \frac{(1-\theta)\lambda_{\min}(B_aA)}{J}.
\end{equation*}
\end{remark}

\begin{remark}
After $J$ steps, we have
\begin{equation*}
E(\| e^{J} \|^2_A) \leq  \left(1 - \frac{(1-\theta)\lambda_{\min}(B_aA)}{J} \right)^J || e^0||_A^2,
\end{equation*}
and the energy error reduction is bounded by
\begin{equation*}
\left(1 - \frac{(1-\theta) \lambda_{\min}(B_aA)}{J} \right)^J \approx 
\exp \left( (\theta-1) \lambda_{\min}(B_aA) \right), \quad 0\le \theta <1.
\end{equation*}
\end{remark}

The following corollary consider a special case of Theorem \ref{thm:randomSSC-fault} that  all the subspace corrections are exact, i.e, $R_i = A_i^{-1}$.  
\begin{corollary} \label{thm:randomSSC-fault-exact}
Assume that errors occurs with probability $\theta \in[0,1)$ which is independent of $k$, then Algorithm \ref{alg:randomSSC-fault} with $R_i = A_i^{-1}$ converges with the expected error decay rate,
\begin{align} 
E(\|u - u^{k+1}\|_A^2) &= \left(1 -  \frac{(1-\theta)\delta_k}{J} \right) E( \| u-u^{k}\|_A^2)  \nonumber\\
& =  \prod_{\ell=0}^{\ell=k}\left(1 - \frac{(1-\theta)\delta_{\ell}}{J}\right) \| u - u^0 \|_A^2, \label{eqn:randomSSC-1-exact}
\end{align}
where $0 < \delta_k = E (\sum_{i=1}^J ||P_i e^k ||_A^2)/ E(||e^k||_A^2) < J$.
\end{corollary}

Next theorem shows that the fault-tolerant randomized SSC method converges almost surely or with probability $1$ if all the subspace corrections are convergent.  
 
\begin{theorem} \label{thm:randomSSC-fault-p1}
Assume that $\| I - T_i \|_A < 1$ for $i=1,2, \cdots, J$. Moreover, assume that errors occur with probability $\theta \in[0,1)$ and independent of $k$, Algorithm \ref{alg:randomSSC-fault} converges almost surely or with probability $1$, i.e., $\| e^k \|_A^2 \overset{a.s}{\longrightarrow} 0$.
\end{theorem}
\begin{proof}
Following the proof  of Theorem \ref{thm:randomSSC-1-p1},  we have
\begin{align*}
P(\| e^k \|_A^2  \geq \varepsilon) \leq \frac{1}{\varepsilon} E(\| e^k \|_A^2) = \frac{1}{\varepsilon} \prod_{\ell=0}^{k-1} \left( 1 - \frac{(1-\theta)\delta_{\ell}}{J} \right) \| e^0 \|_A^2.
\end{align*}
Since $\| I - T_i \|_A < 1$, therefore, we have $\lambda_{\min}<\delta_{\ell} < \lambda_{\max} < J$ and,
\begin{align*}
\sum_{k=1}^{\infty} \frac{1}{\varepsilon} \prod_{\ell=0}^{k-1} \left( 1 - \frac{(1-\theta)\delta_{\ell}}{J} \right) \| e^0 \|_A^2 & \leq \frac{1}{\varepsilon} \| e^0 \|_A^2 \sum_{k=1}^{\infty} \left( 1 - \frac{(1-\theta)\lambda_{\min}(B_aA)}{J}  \right)^{k-1} \\
&= \frac{1}{\varepsilon} \| e^0 \|_A^2 \frac{J}{(1-\theta)\lambda_{\min}(B_aA)} < + \infty.
\end{align*}
Therefore, we have $\sum_{k=1}^{\infty} P(\| e^k \|_A^2  \geq \varepsilon) < + \infty $ which completes the proof. 
\end{proof}

Theorem \ref{thm:randomSSC-fault} and \ref{thm:randomSSC-fault-p1} suggest that, when error occurs, simply do not update the solution or reject the update.  The randomized SSC method are guaranteed to converge to the correct solution.  Such property is useful for developing error resilience algorithms.

\section{Application of Markov Chains}
In this section, we try to explain the theory in terms of the discrete time Markov chains which may allow more general consideration of the randomized method of subspace correction.

Let us first consider Algorithm~\ref{alg:randomSSC-1}.  In this case, $\Omega = \{ 1, 2, \cdots, J \}$ is considered as the \emph{state space} and choosing an index $i \in \Omega$ at each iteration $k$ gives a sequence $\{ S_k \}_{k \geq 0}$ of random variables with value in the set $\Omega$.  Because we choose index $i$ independent of $k$ at each step, for every sequence $i_0, i_1, \cdots, i_{k-1}, i_k, i_{k+1}$ of elements of $\Omega$, $k > 0$, we have the following \emph{Markov property} holds for $\{ S_k \}_{k \geq 0}$,
\begin{equation*}
P(S_{k+1} = i_{k+1}| S_{k} = i_k, S_{k-1} = i_{k-1}, \cdots, S_{0} = i_0) = P(S_{k+1} = i_{k+1}| S_{k} = i_k).
\end{equation*}
Moreover, since we choose index uniformly at each iteration, we have,
\begin{equation*}
P(S_{k+1} = i_{k+1}| S_{k} = i_k) = \frac{1}{J}.
\end{equation*}
This gives the \emph{transition matrix} $\mathbf{P} \in \mathbb{R}^{J \times J}$ such that
\begin{equation*}
p_{ij} = P(S_{k+1} = j | S_{k} = i) = \frac{1}{J}.
\end{equation*}
Note that 
\begin{equation*}
\mathbf{P} = \frac{1}{J} \mathbf{1} \mathbf{1}^T, \ \text{and} \ \mathbf{P}^n = \mathbf{P}, \ n = 2, 3, \cdots
\end{equation*}
where $\mathbf{1} = (1,1, \cdots, 1)^T$.  Based on those definitions, we can see that the sequence $\{ S_k \}_{k \geq 0}$ is indeed a \emph{Markov chain}.

Now we try to prove Theorem~\ref{thm:randomSSC-1} again using the language of Markov chain.
\begin{proof}[Proof of Theorem~\ref{thm:randomSSC-1}]
As before, we define functions 
\begin{equation*}
g_k(i_k) := \| (I - T_{i_k}) e^{k-1} \|_A^2, \  i_k \in \Omega,
\end{equation*}
and the composition of $g_k$ and $S_k$ define a random variable $X_k = g_k \circ S_k$. Denote $g_k^{i_k} := g_k(i_k)$, for $i_k \in \Omega$, we can compute the conditional expectation as following
\begin{align*}
E(X_{k+1}| X_k)(g_k^{i_k}) &= E(X_{k+1} | X_k = g_k^{i_k}) \\
& = \sum_{i_{k+1} \in \Omega} g_{k+1}^{i_{k+1}} \ P(X_{k+1} = g_{k+1}^{i_{k+1}}  | X_k = g_k^{i_k}) \\
& = \sum_{i_{k+1} \in \Omega} g_{k+1}^{i_{k+1}} \ P(S_{k+1} = i_{k+1}  | S_k = i_k)\\
& = \sum_{i_{k+1} \in \Omega} \| (I - T_{i_{k+1}}) e^k \|_A^2 \ \frac{1}{J} \\
& = \sum_{i \in \Omega} \| (I - T_i) e^k \|_A^2 \ \frac{1}{J} \\
& = \| e^k \|_A^2 - \frac{1}{J}(B_aA e^k, e^k)_A.
\end{align*}
Then apply $E(X) = E(E(X|Y))$ as before, we finish the proof. 
\end{proof}

\section{Conclusion} \label{sec:conclu}
We study the convergence behavior of the randomized subspace correction methods.  In stead of the usual upper bound for the convergence rate, we derived an identity for the estimation of the expect error decay rate in energy norm and also show the the randomized algorithm converges almost surely if all the subspace correction converges. 

We also propose another version randomized subspace correction method in which each subspace is corrected once within $J$ iterations.  We theoretically prove that it is convergent by using the XZ-identity and show how it improves the standard SSC method at the worst case in terms of the convergence rate.  

In order to improve the error resilience of the subspace correction methods, we develop a fault-tolerant variant of the randomized method by rejecting any correction when error occurs.  We show that the fault-tolerant iterative method based on such approach converges with probability $1$ if all the subspace corrections are convergent and, moreover, we also derive a sharp identity estimate for the convergence rate.  These results show the intrinsic fault-tolerant features of the subspace correction method and its potential in extreme-scale computing by introducing randomization.


\bibliographystyle{spmpsci}      
\bibliography{RandomMSC}   

\begin{thebibliography}{10}
\providecommand{\url}[1]{{#1}}
\providecommand{\urlprefix}{URL }
\expandafter\ifx\csname urlstyle\endcsname\relax
  \providecommand{\doi}[1]{DOI~\discretionary{}{}{}#1}\else
  \providecommand{\doi}{DOI~\discretionary{}{}{}\begingroup
  \urlstyle{rm}\Url}\fi

\bibitem{Bramble.J1993}
Bramble, J.: Multigrid methods.
\newblock Chapman \& Hall/CRC (1993)

\bibitem{Cui.T;Xu.J;Zhang.C2017a}
Cui, T., Xu, J., Zhang, C.S.: An error-resilient redundant subspace correction
  method.
\newblock Computing and Visualization in Science \textbf{18}(2-3), 65--77
  (2017)

\bibitem{Eldar.Y;Needell.D.2011a}
Eldar, Y.C., Needell, D.: {Acceleration of randomized Kaczmarz method via the
  Johnson--Lindenstrauss Lemma}.
\newblock Numerical Algorithms \textbf{58}(2), 163--177 (2011)

\bibitem{Griebel.M;Oswald.P.1995a}
Griebel, M., Oswald, P.: On the abstract theory of additive and multiplicative
  {S}chwarz algorithms.
\newblock Numer. Math. \textbf{70}(2), 163--180 (1995).
\newblock \doi{10.1007/s002110050115}.
\newblock \urlprefix\url{http://dx.doi.org/10.1007/s002110050115}

\bibitem{Griebel.M;Oswald.P.2012a}
Griebel, M., Oswald, P.: Greedy and randomized versions of the multiplicative
  schwarz method.
\newblock Linear Algebra and its Applications \textbf{437}(7), 1596--1610
  (2012)

\bibitem{Hackbusch.W.1985a}
Hackbusch, W.: Multigrid methods and applications, \emph{Springer Series in
  Computational Mathematics}, vol.~4.
\newblock Springer-Verlag, Berlin (1985)

\bibitem{Hoemmen.M;Heroux.M.2011a}
Hoemmen, M., Heroux, M.A.: Fault-tolerant iterative methods via selective
  reliability.
\newblock In: Proceedings of the 2011 International Conference for High
  Performance Computing, Networking, Storage and Analysis (SC). IEEE Computer
  Society, vol.~3, p.~9 (2011)

\bibitem{Leventhal.D;Lewis.A.2010a}
Leventhal, D., Lewis, A.S.: Randomized methods for linear constraints:
  Convergence rates and conditioning.
\newblock Mathematics of Operations Research \textbf{35}(3), 641--654 (2010)

\bibitem{Liu.J;Wright.S.2013a}
Liu, J., Wright, S.J.: An accelerated randomized kaczmarz algorithm.
\newblock arXiv preprint arXiv:1310.2887  (2013)

\bibitem{Mansour.H;Yilmaz.O.2013a}
Mansour, H., Yilmaz, O.: A fast randomized kaczmarz algorithm for sparse
  solutions of consistent linear systems.
\newblock arXiv preprint arXiv:1305.3803  (2013)

\bibitem{Needell.D.2010a}
Needell, D.: Randomized kaczmarz solver for noisy linear systems.
\newblock BIT Numerical Mathematics \textbf{50}(2), 395--403 (2010)

\bibitem{Nesterov.Y.2012a}
Nesterov, Y.: Efficiency of coordinate descent methods on huge-scale
  optimization problems.
\newblock SIAM Journal on Optimization \textbf{22}(2), 341--362 (2012)

\bibitem{Oswald.P;Zhou.W2015a}
Oswald, P., Zhou, W.: {Convergence analysis for Kaczmarz-type methods in a
  Hilbert space framework}.
\newblock Linear Algebra and its Applications \textbf{478}, 131--161 (2015)

\bibitem{Quarteroni.A;Valli.A1999}
Quarteroni, A., Valli, A.: Domain decomposition methods for partial
  differential equations.
\newblock Oxford University Press, USA (1999)

\bibitem{Richtarik.P;Takac.M.2014a}
Richt{\'a}rik, P., Tak{\'a}{\v{c}}, M.: Iteration complexity of randomized
  block-coordinate descent methods for minimizing a composite function.
\newblock Mathematical Programming \textbf{144}(1-2), 1--38 (2014)

\bibitem{Roy-Chowdhury.A;Banerjee.P.1993a}
Roy-Chowdhury, A., Banerjee, P.: A fault-tolerant parallel algorithm for
  iterative solution of the laplace equation.
\newblock In: Parallel Processing, 1993. ICPP 1993. International Conference
  on, vol.~3, pp. 133--140. IEEE (1993)

\bibitem{Roy-Chowdhury.A;Bellas.N;Banerjee.P.1996a}
Roy-Chowdhury, A., Bellas, N., Banerjee, P.: Algorithm-based error-detection
  schemes for iterative solution of partial differential equations.
\newblock Computers, IEEE Transactions on \textbf{45}(4), 394--407 (1996)

\bibitem{Shantharam.M;Srinivasmurthy.S;Raghavan.P.2012a}
Shantharam, M., Srinivasmurthy, S., Raghavan, P.: Fault tolerant preconditioned
  conjugate gradient for sparse linear system solution.
\newblock In: Proceedings of the 26th ACM international conference on
  Supercomputing, pp. 69--78. ACM (2012)

\bibitem{Southwell.R.1946a}
Southwell, R.V.: Relaxation methods in Engineering Science - A Treatise in
  Approximate Computation.
\newblock Oxford Univ. Press (1946)

\bibitem{Stoyanov.M;Webster.C.2013a}
Stoyanov, M.K., Webster, C.G.: Numerical analysis of fixed point algorithms in
  the presence of hardware faults.
\newblock Tech. rep., Tech. rep. Oak Ridge National Laboratory (ORNL) (2013)

\bibitem{Strohmer.T;Vershynin.R.2009a}
Strohmer, T., Vershynin, R.: {A randomized Kaczmarz algorithm with exponential
  convergence}.
\newblock Journal of Fourier Analysis and Applications \textbf{15}(2), 262--278
  (2009)

\bibitem{Toselli.A;Widlund.O2005}
Toselli, A., Widlund, O.: Domain decomposition methods--algorithms and theory.
\newblock Springer Verlag (2005)

\bibitem{Trottenberg.U;Oosterlee.C;Schuller.A2001}
Trottenberg, U., Oosterlee, C., Sch{\"u}ller, A.: Multigrid.
\newblock Academic Pr (2001)

\bibitem{Xu.J1992}
Xu, J.: Iterative methods by space decomposition and subspace correction.
\newblock SIAM Rev. \textbf{34}(4), 581--613 (1992).
\newblock \doi{10.1137/1034116}.
\newblock \urlprefix\url{http://dx.doi.org/10.1137/1034116}

\bibitem{Xu.J;Zikatanov.L2002}
Xu, J., Zikatanov, L.: The method of alternating projections and the method of
  subspace corrections in {H}ilbert space.
\newblock J. Amer. Math. Soc. \textbf{15}(3), 573--597 (2002).
\newblock \doi{10.1090/S0894-0347-02-00398-3}.
\newblock \urlprefix\url{http://dx.doi.org/10.1090/S0894-0347-02-00398-3}

\end{thebibliography}


\end{document}